\documentclass[12pt,leqno,amsfonts,amscd]{amsart}
\usepackage{epsfig}
\usepackage{amsmath}
\usepackage{amsfonts}
\usepackage{amssymb}
\usepackage{color}
\newtheorem{theorem}{Theorem}[section]
\newtheorem{lemma}[theorem]{Lemma}
\newtheorem{prop}[theorem]{Proposition}
\newtheorem{corollary}[theorem]{Corollary}

\theoremstyle{definition}
\newtheorem{definition}[theorem]{Definition}
\newtheorem{example}[theorem]{Example}

\theoremstyle{remark}
\newtheorem{remark}[theorem]{Remark}
\numberwithin{equation}{section}

\newcommand{\DD}{{\mathbb D}}

\newcommand{\RR}{{\mathbb R}}
\newcommand{\CC}{{\mathbb C}}

\newcommand{\eps}{\varepsilon}

 \DeclareMathOperator{\psh}{PSH}
 \DeclareMathOperator{\sh}{SH}

\renewcommand{\phi}{\varphi}

\title{Connectedness in the Pluri-fine Topology}

\author{Said El Marzguioui and Jan Wiegerinck\\}

\begin{document}

\maketitle \footnote{2000 Mathematics Subject Classification
32U15, 31C40, 32U05, 30C85}

\begin{abstract} We study connectedness in the pluri-fine topology on $\CC^n$ and
obtain the following results. If $\Omega$ is a pluri-finely  open
and pluri-finely connected set in $\CC^n$ and $E\subset\CC^n$ is
pluripolar, then $\Omega\setminus E$ is pluri-finely connected.
The proof hinges on precise information about the structure of
open sets in the pluri-fine topology: Let $\Omega$ be a
pluri-finely open subset of $\CC^{n}$. If $z$ is any point in
$\Omega$, and $L$ is a complex line passing through $z$, then
obviously $\Omega \cap L$ is a finely open neighborhood of $z$ in
$L$. Now let $C_L$ denote the finely connected component of $z$ in
$\Omega\cap L$. Then $\cup_{L\ni z} C_L$ is a pluri-finely
connected neighborhood of $z$. As a consequence we find that if
$v$ is a finely plurisubharmonic function defined on a
pluri-finely connected pluri-finely open set, then $v= -\infty$ on
a pluri-finely open subset implies $v\equiv -\infty$.

\end{abstract}
\begin{scriptsize}
\scriptsize {\noindent {\sc \textbf{Key words:}} Fine topology,
Subharmonic functions, Harmonic measure, Pluri-fine topology,\\
Plurisubharmonic functions, Pluripolar sets, Pluripolar hulls}.
\end{scriptsize}
\begin{center}
\section{Introduction}
\end{center}

The pluri-fine topology on an open set $\Omega$ in $\CC^{n}$ is
the coarsest topology on $\Omega$ that makes all plurisubharmonic
(PSH) functions on $\Omega$ continuous. Understanding the
pluri-fine topology is a first step towards understanding
pluri-fine potential theory and pluri-fine holomorphy. There is
now some evidence, see \cite{EEW,EJ}, that a good theory of finely
plurisubharmonic and finely holomorphic functions may be needed
for dealing with questions about pluripolarity.

In order to avoid cumbersome expressions like ``locally
pluri-finely connected sets'', we adopt the following {\bf
convention}: Topological notions referring to the pluri-fine
topology will be qualified by the prefix $\mathcal{F}$ to
distinguish them from those pertaining to Euclidean topology. For
example, $\mathcal{F}$-open, $\mathcal{F}$-domain (it means
$\mathcal{F}$-open and $\mathcal{F}$-connected), $\mathcal{F}$-
component,.... In view of the fact that the pluri-fine topology
restricted to a complex line coincides with the fine topology on
that line, this convention can be used in the one dimensional
setting, where we will work with the fine topology, at the same
time.

In a previous paper \cite{SJ} we proved that the pluri-fine
topology is locally connected, and we stated the following
theorem.

\begin{theorem}\label{th1} Let $U$ be an $\mathcal{F}$-domain
in $\subseteq \CC^{n}$. If $E$ is a pluripolar set, then $U
\backslash E$ is $\mathcal{F}$-connected.
\end{theorem}

We referred to the corresponding result in fine potential theory
for a proof, but this is unjustified, as Norman Levenberg
\cite{Le} noticed. Here we will give a proof of Theorem \ref{th1}.
It will be a consequence of technical result, Proposition
\ref{pr2}. A slightly weaker but easy formulation is as follows.
For a point $z$ in an $\mathcal{F}$-open subset $\Omega\subset
\CC^n$ and $L$ a complex line passing through $z$, denote by $C_L$
the $\mathcal{F}$-component of $z$ in the $\mathcal{F}$-open set
$\Omega\cap L$.

\begin{theorem}\label{th2}
Let $\Omega$ be an $\mathcal{F}$-open subset of $\CC^{n}$ and let
$z\in\Omega$. Then $\cup_{L\ni z} C_L$ is an
$\mathcal{F}$-neighborhood of $z$ which is
$\mathcal{F}$-connected.
\end{theorem}

Note that $C_L$ is $\mathcal{F}$-open in $L$, because the fine
topology is locally connected, cf.~\cite{Fu71}. We will also
present here (cf. Corollary \ref{co2}) a different proof of this
fact. Theorem \ref{th2} includes the main result in \cite{SJ}:

\begin{corollary}\label{co1} The pluri-fine topology on an open set $\Omega$
in $\CC^{n}$ is locally connected.
\end{corollary}
The proof of the local connectedness in the present paper is
conceptually easier, but uses much more information on the
structure of $\mathcal{F}$-open sets, whereas the proof in
\cite{SJ} reveals that one can find an explicit
$\mathcal{F}$-neighborhood basis consisting of
$\mathcal{F}$-domains, cf. Remark \ref{rema1} below.

In Section \ref{final} we give a definition of
$\mathcal{F}$-plurisubharmonic functions and obtain the following
result.
\begin{theorem}\label{th3}
Let $f$ be an $\mathcal{F}$-plurisubharmonic function on an
$\mathcal{F}$-domain $\Omega$. If $f=-\infty$ on an
$\mathcal{F}$-open subset of $\Omega$, then $f\equiv-\infty$.
\end{theorem}
In fine potential theory a much more precise result holds (cf.
\cite{Fu72}, Theorem 12.9). Nevertheless, Theorem \ref{th3} will
turn out to be very useful. Besides its key role in the proof of
Theorem \ref{th1}, it entails an interesting consequence for the
study of pluripolar hulls of graphs. Namely, the main result of
Edlund and J\"oricke in \cite{EJ} can be extended to functions of
several complex variables. We will discuss this in Section
\ref{final1}.

\section{Preliminaries}
We fix the following notation: $\DD(a,r)=\{|z-a|<r\}$,
$\DD=\DD(0,1)$, $C(a,r)=\{|z-a|=r\}$, while
$B(a,R)=\{\|z-a\|<R\}\subset\CC^n$.
\subsection{Harmonic measure}
Let $\Omega$ be an open set in the complex plane $\CC$ and let $E
\subseteq \overline{\Omega}$. Subharmonic functions on $\Omega$
are denoted by $\sh(\Omega)$, while
$\sh^-(\Omega)=\{u\in\sh(\Omega): \ u\le 0\}$. The harmonic
measure (or the relative extremal function) of $E$ (relative to
$\Omega$) at $z \in \Omega$ is defined as follows (see,
e.g.~\cite{Ra})
$$
\omega(z, E, \Omega)= \sup\{u(z): \ u \in \sh_-(\Omega),\
\limsup_{\Omega\ni v \rightarrow \zeta }u(v)\leq -1 \ \text{for} \
\zeta \in E\}.
$$
This function need not be subharmonic in $\Omega$, but its upper
semi-continuous regularization
$$
\omega(z, E, \Omega)^{*}=\limsup_{\Omega\ni v \rightarrow z
}\omega(v, E, \Omega)\geq \omega(v, E, \Omega)
$$
is subharmonic. If $E$ is a closed subset of $\Omega$, then
$\omega(. , E, \Omega)$ coincides with the Perron solution of the
Dirichlet problem in $\Omega \backslash E$ with boundary values
$-1$ on $\partial E \cap \Omega$ and $0$ on $\partial \Omega
\backslash \partial E$. \\

Recall the following result, cf. \cite{Bre} and \cite{Cho}.
\begin{theorem}\label{th5} Let $\Omega $ be a bounded open subset of
$\CC$. If $E \subset \Omega $ is a Borel set, then there exists an
increasing sequence of compact sets $K_{j} \subset E$ such that
$\omega(z, K_{j}, \Omega)^{*} \downarrow \omega(z, E,
\Omega)^{*}$.
\end{theorem}

Let $E\subset\DD$. We associate to $E$ its circular projection
$$
E^{\circ}=\{|z|: \ z \in E \}.
$$
There is extensive literature on harmonic measure and its behavior
under geometric transformations such as projection,
symmetrization, and polarization. We refer to \cite{Ra} and the
survey article \cite{Bet} and the references therein.

Our main tool in this paper is the following classical theorem of
A. Beurling and R. Nevanlinna related to the Carleman-Milloux
problem, cf. \cite{Beu} and \cite{Ne33}. See also \cite{Bet}.

\begin{theorem}\label{th4} Let $F\subset\DD$ be compact. Let
$F^{\circ}$ be its circular projection. Then
$$
\omega(z, F, \DD)\leq \omega(-|z|, F^{\circ}, \DD), \ \text{for \
all} \  z\in \DD\backslash F.
$$
\end{theorem}


We will need the following result, observed in \cite{BT}, cf.
\cite{SJ}, Lemma 3.1.
\begin{theorem}\label{th6}Sets of the form
$$
\Omega_{B(z, r) ,  \varphi ,c} = \{w\in B(z, r) :\varphi (w)> c\},
$$
where $\varphi \in \psh(B(z, r))$ and $c \in \RR$, constitute a
base of the pluri-fine topology on $\CC^{n}$.
\end{theorem}
\begin{remark}\label{rema1} It follows from the proof of Theorem 1.1 in
\cite{SJ} that each point in an $\mathcal{F}$-open set has a
neighborhood basis consisting of $\mathcal{F}$-domains of the form
$\Omega_{B(z, r), \varphi ,c}$.
\end{remark}

\section{Estimates for Subharmonic Functions}
\begin{lemma}\label{le0} For every $d<c<0$ there exists $\kappa>0$ such that
for every $\varphi \in \sh_{-}(\DD)$ with $\varphi(0)> c$ and for
every point $a$ in the $\mathcal{F}$-open set $$V=\{z \in \DD(0,
1/8): \ \varphi (z)> c\}$$ the set
$$\Omega= \{z\in \DD: \ \varphi (z)\geq d\}$$
contains a circle $C(a,\delta_{\varphi , a})$ with radius
$\delta_{\varphi , a}>\kappa$.
\end{lemma}
\begin{proof} After multiplying $\varphi$ by a constant we can assume that
$d=-1$. Moreover, we may assume that the set $E=\{z\in \DD: \
\varphi (z)< d\}$ is non empty since otherwise the lemma trivially
holds.

Let $a \in V$ be fixed. We will first prove the following estimate
\begin{equation}\label{X}
\varphi (a)\leq \omega(a, E_{a}^{\circ}, \DD(a, 3/4))^{*},
\end{equation}
where $E_{a}^{\circ}= \{a+|z-a|: \ z \in E \cap \DD(a, 3/4) \}$.

Let $f$ be the function $f(z)=z+a$. Note that the circular
projection commutes with $f^{-1}$, i.e.,
$f^{-1}(E_{a}^{\circ})=(f^{-1}(E \cap \DD(a, 3/4)))^{\circ}$.
Hence, to prove (\ref{X}) it is enough, in view of the conformal
invariance of the harmonic measure, to prove that the estimate
(\ref{X}) holds for the particular point $a=0$, i.e.,
\begin{equation}\label{XX}
\varphi (0)\leq \omega(0, E^{\circ}, \DD(0, 3/4))^{*}.
\end{equation}
By Theorem \ref{th5}, there is an increasing sequence of compact
subset $K_{j}$ of $E^{\circ}$ such that
\begin{equation}\label{XXX}
\omega(0, K_{j}, \DD(0, 3/4))^{*} = \omega(0, K_{j}, \DD(0, 3/4))
\downarrow \omega(0, E^{\circ}, \DD(0, 3/4))^{*}.
\end{equation}
The equality in (\ref{XXX}) holds because $K_{j}$ is compact and
$0\notin K_{j}$. Let $\varepsilon > 0$. It follows from
(\ref{XXX}) that there exists a natural number $j_{0}$ such that
\begin{equation}\label{Y}
\omega(0, K_{j_{0}}, \DD(0, 3/4))\leq \omega(0, E^{\circ}, \DD(0,
3/4))^{*}+\varepsilon.
\end{equation}
We can find a compact set $L \subset E$ such that $L^{\circ} =
K_{j_{0}}$. By Theorem \ref{th4} together with inequality
(\ref{Y}) we get
\begin{equation}\label{YY}
\omega(0, L, \DD(0, 3/4))\leq \omega(0, K_{j_{0}}, \DD(0,
3/4))\leq \omega(0, E^{\circ}, \DD(0, 3/4))^{*}+ \varepsilon.
\end{equation}
Since $L \subset E$, and $\varphi (z) < -1$, for all $z \in E$,
inequality (\ref{YY}) implies the following estimate
\begin{equation}\label{Z}
\varphi (0) \leq \omega(0, E^{\circ}, \DD(0, 3/4))^{*}+
\varepsilon.
\end{equation}
As $\varepsilon$ is arbitrary, the estimate (\ref{XX}), and
therefore also (\ref{X}), follows.

Let now $\alpha \in ]0, 1/4[ $ be a constant such that $$I= \{z
\in \DD(a, 3/4): \ \Im z = \Im a, \ \text{and} \ \Re a+ \alpha\leq
\Re z \leq 1/2 \}\subset E_{a}^{\circ}.$$ Then by (\ref{X})
\begin{equation}\label{ZZ}
\varphi (a)\leq \omega(a, E_{a}^{\circ}, \DD(a, 3/4))^{*} \leq
\omega(a, I, \DD(a, 3/4)).
\end{equation}
Again by the conformal invariance, (\ref{ZZ}) yields
\begin{equation}\label{ZZZ}
\varphi (a)\leq \omega(0, f^{-1}(I), \DD(0, 3/4)).
\end{equation}
Since $f^{-1}(I)= [\alpha, 1/2-\Re a]$, it follows that
\begin{equation}\label{A}
\varphi (a)\leq \omega(0, [\alpha, 3/8], \DD(0, 3/4)).
\end{equation}
Let $\alpha_j \downarrow 0$ be a sequence decreasing to $0$. Since
$\alpha_j \rightarrow \omega(0, [\alpha_{j}, 3/8], \DD(0, 3/4))$
decreases to $-1$ (see e.g \cite{H69}, Theorem 8.38), there exists
a constant $0 < \kappa < 3/8$ depending only on $c$ but not on the
function $\varphi$ such that
\begin{equation}\label{AA}
\omega(0, [\kappa, 3/8], \DD(0, 3/4)) < c.
\end{equation}
The last inequality together with (\ref{A}) hence show that for
all $a \in V$,  the interval $I= \{z \in \DD(a, 3/4): \ \Im z =
\Im a, \ \text{and} \ \kappa+\Re a \leq \Re z \leq 3/8 \}$ can not
be a subset of $E_{a}^{\circ}$. We conclude that there exists a
$\delta_{\varphi , a} \in [\kappa, 1/2]$ such that
\begin{equation}\label{AAA}
\{z : \ |z-a|=\delta_{\varphi , a}\} \subset \Omega= \{z\in \DD: \
\varphi (z)\geq d\}.
\end{equation}
\end{proof}

For our purposes we don't need precise estimates for $\kappa$, but
these can be easily obtained using the formula of the harmonic
measure of an interval, cf.~\cite{Bar}.

\begin{lemma}\label{le2} Let $\varphi \in \sh(\DD)$ such that $0\leq \varphi \leq
1$. Let $U$ be the $\mathcal{F}$-open subset of $\DD$ where
$\varphi> 0$. Suppose that there exists a piecewise-$C^1$ Jordan
curve $\gamma\subset \DD$  such that $\gamma \subset U$. Let
$\Gamma$ be the bounded component of the complement of $\gamma$.
Then $W=U\cap \Gamma$ is polygonally connected and, hence
$\mathcal{F}$-connected.
\end{lemma}

\begin{proof} We follow Fuglede's ideas in \cite{Fu75}, section 5.
A square shall be an open square with sides parallel to the
coordinate axes. The square centered at $z$ with diameter $d$ will
be denoted by $Q(z,d)$, its boundary by $S(z,d)$.

Let $n \geq 1$ be a natural number. For every $z\in\gamma$ there
exists $0<d_z< 1/n$ such that $\varphi>\varphi(z)/2$ on
$S(z,d_z)\subset \DD$. This may be proved similarly as the
corresponding well-known statement for circles, cf.~\cite{H69},
Theorem 10.14. The squares $Q(z,d_z)$ cover $\gamma$. By
compactness we can select a finite subcover $\{Q(z_j,d_j),
j=1\dots m_n\}$, that is minimal in the sense that no square can
be removed without loosing the covering property. Now
$\Omega_n=\cup_{j=1}^{m_n} Q(z_j,d_{j})$ is an open neighborhood
of $\gamma$, the boundary of which is contained in
 $$\{\varphi> \frac12\min_{1\le j\le m_n} \{\varphi(z_j)\}\}.$$
Since $\gamma $ is locally connected, the boundary of $\Omega_n$
will consist of two polygonal curves if $n$ is sufficiently large.
One of these components, say, $\gamma_n$ is contained in $\Gamma$.
Denote by $\Gamma_n$ the bounded component of the complement of
$\gamma_n$.

Let $0<\eps<\frac12\min_{1\le j\le m_n} \{\varphi(z_j)\}$ and
$K^n_\eps=\{\varphi\ge\eps\}\cap \bar \Gamma_n$. Then $K^n_\eps$
is a compact subset of $\DD$. Since $\partial\Gamma_n$(=
$\gamma_n$) is contained in $K^n_\eps$, an easy application of the
maximum principle shows that $K^n_\eps$ is connected. Let $z_1,
z_2$ be points in $K^n_\eps$. Repeating the above argument, we
find for every $\delta>0$ a polygonal curve $C$ contained in
$K^n_{\eps/2}$, such that $d(z_1,C), d(z_2,C)<\delta$. A well
known result, cf.~\cite{Ra} states that for $z\in \DD$ and almost
all $\theta\in[0,2\pi]$
$$\lim_{r\to0}\varphi(z+re^{i\theta})=\varphi(z).$$
The conclusion is that there exists a polygonal line in
$K^n_{\eps/2}$ that connects $z_1$ with $z_2$. Letting $\eps\to 0$
we conclude that $U \cap \bar \Gamma_n =\cup_{\eps>0} K^n_\eps$ is
polygonally connected. Since every interval is
$\mathcal{F}$-connected, cf. \cite{Fu71} , so is $U \cap \bar
\Gamma_n$. Finally, since $W= \bigcup_{n\geq 1}U \cap \bar
\Gamma_n$, we conclude that $W$ is $\mathcal{F}$-connected.
\end{proof}

As an easy consequence of Lemma \ref{le2} we have the following

\begin{corollary}\label{co2} The fine topology on $\CC$ is locally connected.

\end{corollary}
\begin{proof} Let $z \in \CC$ and let $U \subseteq \CC$ be an $\mathcal{F}$-neighborhood
of $z$. By Theorem \ref{th6} there exists an $\mathcal{F}$-open
$\mathcal{F}$-neighborhood $V=\Omega_{\DD(z,r),\phi,c}\subseteq U$
of $z$. Without loss of generality we may assume that $$ V =
\DD(z,1)\cap \{\varphi > 0\},
$$
as noted before $V$ contains arbitrarily small circles about $z$.
Let $\partial \DD(z, r)$ be one of them. Then by Lemma \ref{le2},
$V \cap \DD(z,r)$ is an $\mathcal{F}$-neighborhood of $z$ which is
an $\mathcal{F}$-domain.
\end{proof}

\begin{remark} Besides the elementary proof that we presented here there are at
least three proofs of this corollary. The first one was found by
Fuglede \cite{Fu71}, who gave a second proof in \cite{Fu72}, page
92. Fuglede \cite{Fu05} observed, furthermore, that since our
proof of the local connectedness in \cite{SJ} does not use the
fact that the fine topology on $\CC =\RR^{2}$ is locally
connected, it provides of course (for $n=1$) a third proof of that
fact.
\end{remark}

\section{Structure of $\mathcal{F}$-Open Sets}
We start this Section with the technical result that was alluded
to in the introduction.
\begin{prop}\label{pr2} Let $U \subseteq \CC^{n}$ be an $\mathcal{F}$-open subset and let $a \in U $.
Then there exists a constant $\kappa = \kappa(U, a)$ and an
$\mathcal{F}$-neighborhood $V\subset U$ of $a$ with the property
that for any complex line $L$ through $v\in V$ the
$\mathcal{F}$-component of the $\mathcal{F}$-open set $U \cap L$
that contains $v$, contains a circle about $v$ with radius at
least $\kappa$.
\end{prop}
\begin{proof} Let $a \in U$. By Theorem \ref{th6} there
exist two constants $r>0$, $d<0$ and a plurisubharmonic function
$\varphi \in \psh(B(a, r))$ such that $\Omega = \{z\in B(a, r)
:\varphi (z)\geq d\}$ is an $\mathcal{F}$-neighborhood of $a$
contained in $U$. Since the pluri-fine topology is
biholomorphically invariant, there is no loss of generality if we
assume that $r=2$, $\varphi \leq0$, $a=0$ and $\varphi (0) = d/2
$. Let
$$
V=\{z\in B(0, 1/8): \ \varphi (z) > d/2\}.
$$
Let $v \in V$ and let $L$ be a complex line through $v$. The
restriction $\varphi_L$ of $\varphi$ to $B(v, 1)\cap L$ is
subharmonic and satisfies the conditions of Lemma \ref{le0}.
Consequently, there exists a constant $\kappa$ depending only on
$d$, but not on $\varphi_L$, such that the set $\{z\in B(v, 1): \
\varphi (z) \geq d\}\cap L$, and therefore $U \cap L$, contains a
circle with radius $\delta_{\varphi_L, v} \in(\kappa,\frac{1}{2})$
about $v$. It follows from Lemma \ref{le2} that the set $U \cap L
\cap B(v, \delta_L)$ is $\mathcal{F}$-connected. This completes
the proof of the proposition.
\end{proof}

\begin{proof}[Proof of Theorem \ref{th2}] Let $V \subseteq \Omega$
be an $\mathcal{F}$-neighborhood of $z$ provided by Theorem
\ref{th6}.
 Without loss of generality we may assume
 $$V=B(z,1)\cap \{\varphi > 0\},$$
for some $\varphi \in \psh(B(z,1))$. Recall that for a complex
line $L$ through $z$, $C_L$ is the $\mathcal{F}$-component of
$\Omega\cap L$ that contains $z$. It is immediate that $\cup_{z\in
L}C_L$ is $\mathcal{F}$-connected. We denote by $\tilde{C_L}$ the
$\mathcal{F}$-component of $V\cap L$ that contains $z$. By Lemma
\ref{le0} together with Lemma \ref{le2} we can find a constant
$\kappa >0$ such that $V \cap B(z,\kappa)\cap L \subseteq
\tilde{C_L}$. As $\tilde{C_L}$ is clearly contained in $C_L$, $V
\cap B(z,\kappa)$ is a subset of $\cup_{L\ni z} C_L$. This proves
that $\cup_{L\ni z} C_L$ is an $\mathcal{F}$-neighborhood of $z$.
\end{proof}

The next gluing lemma was used in \cite{SJ}. The lemma is actually
an immediate consequence of Fuglede's results, but it seemed
interesting to find a proof that avoids the use of the fine
potential theory machinery. Here we give a short direct proof.

\begin{lemma}\label{le5} Let $v\in \sh(D)$ for some domain $D\subset\CC$.
Suppose that $v\ge 0$ and that there exist nonempty, disjoint
$\mathcal{F}$-open sets $D_1,D_2\subset D$ such that
$$\{v>0\}=D_1\cup D_2.$$
Then the function $v_1$ defined by
\begin{equation} v_1(z)=
    \begin{cases}
        0 & \text{if $z\in D\setminus D_1$,}\cr
        v(z)&\text{if $z\in D_1$,}
    \end{cases}
\end{equation}
is subharmonic in $D$.
\end{lemma}

\begin{proof}
For $\eps>0$ let $ D_i(\eps)=D_i\cap \{v\ge\eps\}$, ($i=1,2$). We
claim that $D_1(\eps)$ is closed in $D$. Indeed, take a sequence
$\{x_n\}$ in $D_1(\eps)$ that converges to $y \in D$. Since $ \{v
\geq \eps\}$ is closed in $D$, $v(y)\geq \eps$. Thus $y \in D_1
\cup D_2$. Suppose that $y \in D_2$. Again there exists $r>0$ such
that $C(y,r)$ is contained in the $\mathcal{F}$-open set $\{v >
\eps /2\}$. By Lemma \ref{le2} the set
$$U=\DD(y,r)\cap \{v > \eps /2\} $$ is an $\mathcal{F}$-connected subset of $D_1\cup D_2$.
Since $U \cap D_2$ is non-empty, $U \cap D_1 = \emptyset$. This
contradicts the fact that $U$ contains $x_n$ for $n$ sufficiently
large. Thus $y\in D_1$ and hence $D_1(\eps)$, which proves the
claim.

Now define
\begin{equation} v_\eps(z)=
    \begin{cases}
        \eps & \text{if $z\in D\setminus D_1(\eps)$,}\cr
        v(z)&\text{if $z\in D_1(\eps)$.}
    \end{cases}
\end{equation}

The function $v_\eps$ is clearly upper semicontinuous in $D$ and
it satisfies the mean value inequality in $D\setminus D_1(\eps)$.
Let $a\in D_1(\eps)$ and denote by $\overline{v}(a, r)$ the mean
value of $v$ over the circle $C(a,r)$. Since $D_2(\eps)$ is
similarly closed, $v\leq v_\eps$ on $C(a,r)$ for sufficiently
small $r$. Consequently,
$$v_\eps(a)=v(a)\le \overline{v}(a, r)\le\overline{v_\eps}(a, r).$$
This proves that $v_\eps$ is subharmonic in $D$. Finally, the
sequence $\{v_{1/n}\}_n$ decreases to $v_1$, showing that $v_1$ is
subharmonic.
\end{proof}

It was proved by Gamelin and Lyons in \cite{GL83} that an
$\mathcal{F}$-open subset of $\mathbb{C}$ is
$\mathcal{F}$-connected if and only if it is connected with
respect to the usual topology on $\mathbb{C}$. The next example
shows that this result has no analog in $\mathbb{C}^n$ for $n>1$.
\begin{example} There exists an $\mathcal{F}$-open set $U \subset
\mathbb{C}^{2}$, which is connected but not
$\mathcal{F}$-connected. Indeed, consider the set $$\Gamma = \{(x,
y)\in \CC^2: \ y= e^{1/x}, -1\leq x < 0\}.$$ As was proved by the
second author, cf. \cite{W}, one can find a plurisubharmonic
function $\varphi \in \text{PSH}(B(0, 2))$ such that $\varphi
|_\Gamma = -\infty $ and $\varphi (0)= 0$. Let $V= \{ \varphi > -
1/2\}$ and $W=\{\varphi <-1/2\}$. Let $W_1$ be the connected
component of $W$ that contains $\Gamma$, and let $V_1$ the
$\mathcal{F}$-component of $V$ that contains $0$. Since the
pluri-fine topology is locally connected, $V_1$ is
$\mathcal{F}$-open. Of course $V_1$ is connected since it is
already $\mathcal{F}$-connected. Observe now that $U=V_1 \cup W_1$
is $\mathcal{F}$-open and $\mathcal{F}$-disconnected. On the other
hand, since $W_1 \cup \{0\}$ is clearly connected, the
$\mathcal{F}$-open set $U=V_1 \cup W_1 = V_1 \cup (W_1 \cup
\{0\})$ is connected.
\end{example}
\section{Finely Plurisubharmonic Functions}\label{final}

As far as we know there is no generally accepted definition of
$\mathcal{F}$-pluri\-subharmonic functions. The following,
cf.~\cite{Mo03}, seems quite natural.
\begin{definition}\label{defn1}
A function $f$ :\ $\Omega $ $\longrightarrow$ $[-\infty, \infty [$
($\Omega $ is $\mathcal{F}$-open in $\CC^n$) is called {\em
$\mathcal{F}$-plurisubharmonic} if $f$ is $\mathcal{F}$-upper
semicontinuous on $\Omega$ and if the restriction of $f$ to any
complex line $L$ is finely subharmonic or $\equiv -\infty$ on any
$\mathcal{F}$-component of $\Omega\cap L$.
\end{definition}
It follows immediately from this definition and general properties
of finely subharmonic functions that any usual plurisubharmonic
function is $\mathcal{F}$-pluri\-subharmonic where it is defined.
Moreover, $\mathcal{F}$-plurisubharmonic functions in an
$\mathcal{F}$-open set $\Omega$ form a convex cone, which is
stable under pointwise infimum for lower directed families, under
pointwise supremum for finite families, and closed under
pluri-finely locally uniform convergence. See \cite{Fu72}, page
84-85, and \cite{Mo03}.

Clearly, an $\mathcal{F}$-plurisubharmonic function $f$ on an
$\mathcal{F}$-open set $\Omega$ has an
$\mathcal{F}$-pluri\-subharmonic restriction to every
$\mathcal{F}$-open subset of $\Omega$. Conversely, suppose that
$f$ is $\mathcal{F}$-plurisubharmonic in some
$\mathcal{F}$-neighborhood of each point of $\Omega$. Then $f$ is
$\mathcal{F}$-pluri\-subharmonic in $\Omega$, see.~ \cite{Fu72},
page 70. We shall refer to this by saying that the
$\mathcal{F}$-plurisubharmonic functions have the {\em sheaf
property}.

\begin{theorem}\label{th6a}
Let $f$ be an $\mathcal{F}$-plurisubharmonic function on a
$\mathcal{F}$-domain $\Omega$. If $f=-\infty$ on an
$\mathcal{F}$-open subset $U$ of $\Omega$, then $f\equiv-\infty$.
\end{theorem}

\begin{proof} Without loss of generality we can assume that $U$ is
the $\mathcal{F}$-interior of the set $\{f=-\infty\}$. Let $z_0
\in \Omega$ be an $\mathcal{F}$- boundary point of $U$. After
scaling we can assume that $z_0=0$ and that
\begin{equation}
V=B(0,1)\cap\{\varphi>0\}\subset\Omega.
\end{equation} is an $\mathcal{F}$-neighborhood of
$0$ defined by a PSH-function $\varphi$ on $B(0,1)$ with
$\varphi(0)=1$. Then \begin{equation}
     V_{1/2}=B(0,1/2)\cap\{\varphi>1/2\}
     \end{equation}
is a smaller $\mathcal{F}$-neighborhood of $0$. For every $z\in
V_{1/2}\cap U$ the function $\varphi$ is defined on $B(z,1/2)$ and
$B(z,1/2)\cap \{\varphi>0\}$ is an $\mathcal{F}$-neighborhood of
$z$ contained in $\Omega$.

By Lemma \ref{le0} together with Lemma \ref{le2} there exists
$\kappa>0$ such that for every line $L$ passing through $z\in
V_{1/2}\cap U$ there exists $\delta_{z,L}\in(\kappa,1/2)$ such
that $C_{z,L}=\{\varphi>0\}\cap B(z,\delta_{z,L})\cap L$ is an
$\mathcal{F}$-connected $\mathcal{F}$-neighborhood of $z$ in
$L\cap V$. Because $z\in U$, $C_{z,L}$ meets $U$ in an
$\mathcal{F}$-open subset of $L$. Therefore $f\equiv-\infty$ on
$C_{z,L}$ according to Theorem 12.9 in \cite{Fu72}. It follows
that $f\equiv -\infty$ on the $\mathcal{F}$-open set
$$V\cap B(z,\kappa)\subset\bigcup_{L\ni z} C_{z,L}.$$
Now if $|z|< \kappa$, then $0\in V \cap B(z,\kappa)$. The
conclusion is that the $\mathcal{F}$-boundary of $U$ does not hit
$\Omega $, and therefore $U=\Omega$.
\end{proof}
Since pluripolar sets are $\mathcal{F}$-closed, Theorem \ref{th1}
is a particular case of the following more general result.
\begin{corollary} Let $U$ be an $\mathcal{F}$-domain in $\CC^n$, and
let $E\subset\{f=-\infty\}$, where $f$ is
$\mathcal{F}$-plurisubharmonic on $U$ $(\not\equiv - \infty )$.
Suppose that $E$ is $\mathcal{F}$-closed. Then $U \backslash E$ is
an $\mathcal{F}$-domain.
\end{corollary}
\begin{proof} Suppose that $U\backslash E = V\cup W$, where $V$ and
$W$ are non empty disjoint $\mathcal{F}$-open sets. Define $h$ :
$U \backslash E \rightarrow [-\infty, \infty[$ by

\begin{equation} h(z)=
    \begin{cases}
        0 & \text{if $z\in V$,}\cr
        -\infty&\text{if $z\in W$,}
    \end{cases}
\end{equation}
and define
\begin{equation} \tilde f(z)=
    \begin{cases}
        f +h & \text{if $z\in V\cup W$,}\cr
        -\infty&\text{if $z\in E$.}
    \end{cases}
\end{equation}
Then $\tilde f$ is $\mathcal{F}$-upper semi-continuous. If we
restrict $\tilde f$ to a complex line, it is finely subharmonic.
Indeed, on $V$ because $V\cap L$ is $\mathcal{F}$-open and $\tilde
f=f\not\equiv -\infty$, and in $U\backslash V $ there is nothing
to prove because there $\tilde f=-\infty$. By Theorem \ref{th6a},
$\tilde f\equiv-\infty$, a contradiction
\end{proof}
One more consequence of Theorem \ref{th6a} is the following
maximum principle for $\mathcal{F}$-plurisubharmonic functions.
\begin{theorem} Let $f \leq 0$ be $\mathcal{F}$-plurisubharmonic function on
an $\mathcal{F}$-domain $U$ in $\CC^n$. Then either $f < 0$ or
$f\equiv 0$.
\end{theorem}
\begin{proof} Suppose that the $\mathcal{F}$-open set $ V=\{z\in U \ :
f(z)<0\}$ is not empty. The function $g_n=nf$ is
$\mathcal{F}$-plurisubharmonic. Since $g_n$ decreases on $V$, the
limit function $g$ is $\mathcal{F}$-plurisubharmonic. By Theorem
\ref{th6a}, $g \equiv -\infty $ since it equals $-\infty$ in $V$.
Hence $f < 0$.
\end{proof}
\section{Application to pluripolar hulls}\label{final1}
In this final section we give a definition of
$\mathcal{F}$-holomorphic functions of several complex variables.
Next we prove a higher dimensional analog of Theorem 1.2 in
\cite{EEW}. See also Theorem 1 in \cite{EJ}.
\begin{definition}\label{fdef} Let $U\subseteq \CC^n$ be $\mathcal{F}$-open. A function $f$ :\
$U$ $\longrightarrow$ $\CC$ is said to be
$\mathcal{F}$-holomorphic if every point of $U$ has a compact
$\mathcal{F}$-neighborhood $K\subseteq U$ such that the
restriction $f|_K$ belongs to $H(K)$.
\end{definition}

Here $H(K)$ denotes the uniform closure on $K$ of the algebra of
holomorphic functions in a neighborhood of $K$.

\begin{lemma}\label{flem} Let $U\subseteq\CC^{n}$ be an $\mathcal{F}$-domain, and let $f$ :\
$U$ $\longrightarrow$ $\CC$ be an $\mathcal{F}$-holomorphic
function. Suppose that $h$ :\ $\CC^2$ $\longrightarrow$
$[-\infty,\ +\infty [$ is a plurisubharmonic function. Then the
function $z \mapsto h(z,f(z))$ is $\mathcal{F}$-plurisubharmonic
on $U$ .
\end{lemma}
\begin{proof} First, we assume that $h$ is continuous and finite everywhere.
Let $a\in U$. By Definition \ref{fdef} there is a compact
$\mathcal{F}$-neighborhood $K$ of $a$ and a sequence $f_n$ of
holomorphic functions defined in usual neighborhoods of $K$ that
converges uniformly to $f|_K$. Since $h(z, f_n(z))$ is
plurisubharmonic and converges uniformly to $h(z, f(z))$ on $K$,
$h(z, f(z))$ is $\mathcal{F}$-plurisubharmonic in the
$\mathcal{F}$-interior
of $K$.\\
Suppose that $h$ is arbitrary. Then $h$ is the limit of some
decreasing sequence of continuous plurisubharmonic functions $h_n
\in \psh(\CC^2)$. By the first part of the proof, $(h_n(z,
f(z)))_n$ is a decreasing sequence of
$\mathcal{F}$-plurisubharmonic functions in the
$\mathcal{F}$-interior of $K$. The limit $h(z,f(z))$ is therefore
$\mathcal{F}$-plurisubharmonic in the $\mathcal{F}$-interior of
$K$. Now by the sheaf property of $\mathcal{F}$-plurisubharmonic
function we conclude that $h(z, f(z))$ is
$\mathcal{F}$-plurisubharmonic on $U$.
\end{proof}
A version of this lemma for functions of one variable with similar
proof appears in \cite{EEW}.

The pluripolar hull $E_{\Omega}^*$ of a pluripolar set $E$
relative to an open set $\Omega $ is defined as follows.
$$
E_{\Omega}^{*}=\bigcap \{z\in \Omega \ : u(z)= - \infty\},
$$
where the intersection is taken over all plurisubharmonic
functions defined in $\Omega$ which are equal to $-\infty$ on $E$.

\begin{theorem}\label{thm11} Let $U \subset \CC^{n}$ be an $\mathcal{F}$-domain.
Let $f$ be $\mathcal{F}$-holomorphic in $U$. Suppose that for some
$\mathcal{F}$-open subset $V\subset U $ the graph $ \Gamma_{f}(V)$
of $f$ over $V$ is pluripolar in $\CC^{n+1}$. Then the graph $
\Gamma_{f}(U)$ of $f$ is pluripolar in $\CC^{n+1}$. Moreover, $
\Gamma_{f}(U)\subset (\Gamma_{f}(V))^{\ast}_{\CC^{n+1}}$.
\end{theorem}
\begin{proof} By Josefson's theorem, cf. \cite{Jo78}, there exists
$h\in \psh(\CC^{n+1})$ ($\not\equiv -\infty$) such that $h(z,
f(z))= -\infty$, $\forall z\in V$. In view of Lemma \ref{flem} and
Theorem \ref{th6a}, the function $h(z, f(z))$ is identically
$-\infty$ in $U$. It follows at once that $ \Gamma_{f}(U)$ is
pluripolar and $ \Gamma_{f}(U)\subset
(\Gamma_{f}(V))^{\ast}_{\CC^{n+1}}$.
\end{proof}

As a corollary we obtain a generalization to several complex
variables of the main result of \cite{EJ}. We keep the notation of
Theorem \ref{thm11}.
\begin{corollary} \label{cor1} Suppose that $U$ contains a ball $B$.
Then $ \Gamma_{f}(U)$ is pluripolar and $ \Gamma_{f}(U)\subset
(\Gamma_{f}(B)))^{\ast}_{\CC^{n+1}}$.
\end{corollary}
\begin{proof} On the intersection of $B$ with any complex line $f$ is a $\mathcal{F}$-holomorphic
function of one variable, hence, holomorphic there, cf.
\cite{Fu81}, page 63. Thus $f$ is holomorphic on $B$ and
$\Gamma_{f}(B)$ is pluripolar. Now Theorem \ref{thm11} applies.
\end{proof}

\begin{remark}
Theorem \ref{thm11} and Corollary \ref{cor1} only explain for a
small part the propagation of pluripolar hulls. E.g., in the case
of Corollary \ref{cor1} take $B$ the unit ball and consider the
function $g(z)=f(z)(z_1-z_2^2)$. Then, whatever the extendibility
properties of $f$ may be, the pluripolar hull of graph of $g$ will
contain the set $\{z_1=z_2^2\}$.
\end{remark}

\noindent Kd\textsc{V Institute for Mathematics \\
University of Amsterdam \\
Plantage Muidergracht 24 \\
1018 TV Amsterdam \\
The Netherlands}\\
janwieg@science.uva.nl \\  smarzgui@science.uva.nl

\end{document}